\documentclass[12pt]{article}
\usepackage{amsthm,amsmath,amssymb,
extsizes,MnSymbol,easymat,mathrsfs,etex}

\usepackage[active]{srcltx}
\title{Consimilarity and quaternion
matrix equations
${AX-\hat XB=C}$, ${X-A\hat
XB=C}$\thanks{\emph{Special Matrices} 2 (2014) 180--186\qquad
http://www.degruyter.com/view/j/spma}}

\author{Tatiana Klimchuk\\
Kiev National Taras Shevchenko
University\\
Kiev, Ukraine, klimchuk.tanya@gmail.com
\and
Vladimir V. Sergeichuk%
\\
Institute of Mathematics,
Tereshchenkivska 3\\
Kiev, Ukraine, sergeich@imath.kiev.ua}
\date{}

\renewcommand{\ge}{\geqslant}
\newcommand{\HH}{\mathbb H}
\newcommand{\CC}{\mathbb C}
\newcommand{\RR}{\mathbb R}

\newtheorem{theorem}{Theorem}
\newtheorem{lemma}[theorem]{Lemma}
\newtheorem{corollary}[theorem]{Corollary}

\theoremstyle{definition}

\theoremstyle{remark}

\newtheorem{example}{Example}[section]

\begin{document}
\maketitle

\begin{abstract}
L. Huang [Consimilarity of quaternion
matrices and complex matrices, Linear
Algebra Appl. 331 (2001) 21--30] gave a
canonical form of a quaternion matrix
with respect to consimilarity
transformations $A\mapsto \tilde
S^{-1}AS$ in which $S$ is a nonsingular
quaternion matrix and \[h=a+bi+cj+dk\
\mapsto\ \tilde h:=a-bi+cj-dk\quad
(a,b,c,d\in\RR).\] We give an analogous
canonical form of a quaternion matrix
with respect to consimilarity
transformations $A\mapsto\hat S^{-1}AS$
in which $h\mapsto \hat h$ is an
arbitrary involutive automorphism of
the skew field of quaternions. We apply
the obtained canonical form to the
quaternion matrix equations ${AX-\hat
XB=C}$ and ${X-A\hat XB=C}$.


\end{abstract}

\section{Introduction}

We give a canonical form of a square quaternion matrix with respect to \emph{consimilarity transformations}
$A\mapsto \hat S^{-1}AS$ in which $S$
is a nonsingular quaternion matrix and $\alpha \mapsto \hat \alpha $ is a fixed involutive automorphism of the skew field of quaternions $\HH$, and apply it to the
quaternion matrix equations ${AX-\hat
XB=C}$ and ${X-A\hat XB=C}$.

A canonical form of a square complex
matrix $A$ with respect to {consimilarity transformations}
\begin{equation}\label{jhr}
A\mapsto \bar S^{-1}AS\qquad (S\in\CC^{n\times n}\text{
is nonsingular})
\end{equation}
was given by Hong and Horn \cite{hon}
(see also \cite[Theorem 4.6.12]{j-h}):
$A$ is consimilar to a direct sum,
uniquely determined up to permutation
of direct summands, of matrices of the
following two types:
\begin{equation}\label{fei}
J_k(a)\ (a\in\RR,\ a\ge 0),\quad
\begin{bmatrix}
  0&I_k\\J_k(a+bi)&0
\end{bmatrix}\
  (a,b\in\RR,\ a<0\text{ if }b=0)
\end{equation}
in which
\begin{equation}\label{kyr}
J_k(\lambda ):=
\begin{bmatrix}
  \lambda  &1&&0 \\
  &\lambda &\ddots \\
  &&\ddots&1\\0&&&\lambda
\end{bmatrix}
\end{equation}
is the
$k\times k$ Jordan block with
eigenvalue $\lambda $.

Huang \cite{hua} defined consimilarity
transformations of square quaternion
matrices
\begin{equation}\label{jkg}
A\mapsto \tilde S^{-1}AS\qquad
(S\in\HH^{n\times n}\text{
is nonsingular})
\end{equation}
via the involution
\begin{equation}\label{mjd}
h=a+bi+cj+dk\ \mapsto\
\tilde h:=-jhj=a-bi+cj-dk\quad (a,b,c,d\in\RR);
\end{equation}
on the skew field of quaternions $\HH$;
we suggest to call \eqref{jkg}
\emph{$j$-consimilarity
transformations} since $\tilde S=-jSj$. (The main reason why
consimilarity transformations of
quaternion matrices are not defined via
the quaternion conjugation
$h\mapsto\bar h=a-bi-cj-dk$ is that the
transformations $A\mapsto\bar S^TAS$
are not transitive.)

Huang \cite[Theorem 1]{hua} noticed
that two quaternion matrices
\begin{equation}\label{liy}
\text{$A$ and $B$
are $j$-consimilar}\ \ \Longleftrightarrow
\ \ \text{$jA$ and $jB$
are similar},
\end{equation}
which holds since $\tilde
S^{-1}AS=-jS^{-1}j\cdot AS=B$ if and
only if $S^{-1}jAS=jB$. This statement
admitted him to deduce a canonical form
of quaternion matrices under
$j$-consimilarity from the canonical
form for similarity. Huang's canonical
form given in {\cite[Theorem 3]{hua}}
is a direct sum, uniquely determined up
to permutation of summands, of Jordan
blocks
\begin{equation}\label{jts}
J_k(a+bj)\qquad (a,b\in\mathbb R,\
a\ge 0).
\end{equation}

The definition of consimilarity
transformations via \eqref{mjd} looks
special. We show in Lemma \ref{lih} that all consimilarity
transformations defined via involutive
automorphisms of $\HH$ are transformed from
one to each other by reselections of the
set of orthogonal imaginary units
$i,j,k$. Thus, we can use an involutive
automorphism for which the formulas of the theory of consimilarity transformations are simpler.

Huang \cite{hua} and other
authors (see, for example, \cite{jia,son1,son,song,yua}) study the $j$-consimilarity
transformations \eqref{jkg}. These
transformations act on complex matrices as the consimilarity transformations
\eqref{jhr}. We
suggest to define consimilarity
transformations of quaternion matrices
extending the similarity
transformations of complex matrices.
For this purpose, we define
\emph{$i$-consimilarity
transformations} $\hat S^{-1}AS$ of
quaternion matrices via the
automorphism
\begin{equation*}\label{mmy}
h=a+bi+cj+dk\ \mapsto\
\hat h:=-ihi=a+bi-cj-dk\quad (a,b,c,d\in\RR).
\end{equation*}
Our goal is to show that
$i$-consimilarity transformations are
more convenient for use than
$j$-consimilarity transformations
\eqref{jkg}:
\begin{itemize}
  \item Most properties of
      $j$-consimilarity
      transformations of quaternion
      matrices are closer to
      properties of similarity
      transformations of complex
      matrices than to properties
      of consimilarity
      transformations of complex
      matrices (see \eqref{liy} and
      compare the canonical forms
      given in \eqref{fei} and
      \eqref{jts}).

  \item If $A=U+Vj$, in which
      $U,V\in \HH^{m\times n}$,
      then $\hat A=U-Vj$ (compare
      with $\tilde A=\bar U+\bar
      Vj$ in \cite{yua}), which
      admits us to study similarity
      and $i$-consimilarity
transformations of quaternion
matrices simultaneously in
      Sections \ref{lkv} and
      \ref{jhf}.

  \item The canonical form in
      Theorem \ref{lsd} of a
      quaternion matrix for
      $i$-consimilarity
      transformations $\hat
      S^{-1}AS$ is a complex matrix
      (compare with \eqref{jts});
      in Section \ref{jhf} we apply
      it to the quaternion matrix
      equations ${AX-\hat XB=C}$
      and ${X-A\hat XB=C}$,
      reducing them to complex
      matrix equations.
\end{itemize}

\section{A canonical form for $i$-consimilarity}
\label{lkv}

Consimilarity transformations come from
the theory of semilinear
operators (which is presented, for example, in Jacobson's book \cite[Chapter 3, Section 12]{jac}). It is worth mention that the
fundamental theorem of projective
geometry is formulated in terms of
semilinear maps between vector spaces.

Let $V$ be a right vector space over a field
or skew field $\mathbb F$ with a fixed
automorphism $\alpha \mapsto \hat
\alpha $ on $\mathbb F $. A map
$\mathcal A:V\to V$ is a
\emph{semilinear operator} if
\[\mathcal A(v+w)=\mathcal Av+\mathcal
Aw,\qquad \mathcal A(v\alpha)=(\mathcal Av)\hat\alpha \] for all
$v,w\in V$ and $\alpha \in\mathbb F$. If $A$ is the matrix of $\cal A$ in some basis and $[v]$ is the coordinate vector of $v\in V$, then $[\mathcal Av]=A\widehat{[v]}$.
By transfer to other bases the matrix $A$ is reduced by transformations
\begin{equation}\label{liy1}
A\mapsto C^TA\hat C \qquad (C \text{ is the change of basis matrix).}
\end{equation}
For the most important types of automorphisms
$\alpha \mapsto \hat \alpha $, Jacobson
\cite[Chapter 3, Theorem 34]{jac} reduced the
problem of classifying matrices with respect to transformations \eqref{liy1} to the
problem of classifying matrices with respect to  similarity.

All automorphisms $h\mapsto\hat h$ on
$\HH$ that we consider are
\emph{involutive}; that is, $\Hat{\Hat
h}=h$ for all $h\in \HH$.

\begin{lemma}\label{lih}
If $h\mapsto\hat h$ is an involutive
automorphism of $\HH$, then either it
is identical, or the set of orthogonal
imaginary units $i,j,k$ can be chosen
such that
\begin{equation}\label{cwo}
\hat h= h^{i}:=i^{-1}hi=-ihi=a+bi-cj-dk
\end{equation}
for each $h=a+bi+cj+dk\in\HH$
$(a,b,c,d\in\RR)$.
\end{lemma}

\begin{proof}
By the Skolem--Noether theorem
(\cite[Theorem 2.1]{vig} or
\cite[Theorem 2.41]{kna}), every
automorphism of $\mathbb H$ is of the
form \[q\mapsto q^{\sigma}:=\sigma
^{-1}q\sigma\qquad\text{for some fixed nonzero }
\sigma \in \HH.\] This automorphism is
involutive if and only if $\sigma
^{-2}q\sigma^2=q$ for all $q\in\HH$, if
and only if $\sigma^2\in\mathbb R$.
Multiplying $\sigma $ by a positive
number (which does not change the
automorphism), we get $\sigma^2=\pm 1$.
Write $\sigma=a+b\tau $, in which
$a,b\in\mathbb R$, $\tau\notin\mathbb
R$, and $|\tau|=1$. Since $\tau^2=-1$,
$\sigma ^2=(a^2-b^2)+2ab\tau=\pm 1$.
Hence $ab=0$.

If $\sigma^2=1$, then $b=0$, $a^2=1$,
$\sigma=a=\pm 1$, and the automorphism
$q\mapsto q^{\sigma }$ is the identity.

Let $\sigma^2=-1$. Then $a=0$, $b^2=
1$, $b=\pm 1$, and so $\sigma=\pm
\tau$. Recall that the space of pure quaternions can be identified with the vector space $\RR^3$; the product of pure quaternions $u$ and $v$ can be represented in the form
$
uv= (u,v)+u\times v
$
in which $(u,v)$ is the usual inner (scalar) product and $u\times v$ is the vector cross product.
Reselecting the imaginary units
$i,\,j,\,k$, we set $i:=\tau$, take as
$j$ any pure quaternion such that $|j|=1$ and  $(i,j)=0$ (i.e, $j$ is any imaginary unit that is perpendicular to $i$), put $k:=ij$, and
obtain the automorphism \eqref{cwo}.
\end{proof}

The automorphism \eqref{cwo}
can be written in the form
\begin{equation*}\label{grs}
h=u+vj\ \mapsto\ h^{i}=u-vj\qquad (u,v\in\CC).
\end{equation*}
Thus, $h^{\sigma }=u+\sigma ^2v$ for $\sigma \in\{1,i\}$ and all $h\in\HH$. This admits us to study similarity and
consimilarity transformations
simultaneously using the following
definition.

Let $\sigma \in\{1,i\}$. By
\emph{$\sigma $-consimilarity
transformations of quaternion
matrices} we mean the transformations
\begin{equation*}\label{nhc}
A\mapsto S^{-\sigma}AS:=(\sigma^{-1}S^{-1}
\sigma)
AS=\begin{cases}
S^{-1}AS&\text{if }\sigma =1,\\
-iS^{-1}iAS&\text{if }\sigma =i,
\end{cases}
\end{equation*}
in which $S$ is nonsingular. It suffices to study
$\sigma $-consimilarity transformations
since by Lemma \ref{lih} each involutive automorphism of
$\HH$ has the form $h\mapsto h^{\sigma
}={\sigma }^{-1}h{\sigma }$ ($\sigma
\in\{1,i\}$) in suitable $i,j,k$.

\begin{lemma}[cf. {\cite[Theorem 1]{hua}}]
\label{lit}
The following statements are equivalent
for $A,B\in\HH^{n\times n}$:
\begin{itemize}
  \item[\rm(i)] $A$ and $B$ are
      $i$-consimilar;
  \item[\rm(ii)] $iA$ and $iB$ are
      similar;
  \item[\rm(iii)] $iA$ and $Bi$ are
      similar;
  \item[\rm(iv)] $Ai$ and $Bi$ are
      similar.
\end{itemize}
\end{lemma}

\begin{proof}
(i)$\Leftrightarrow$(ii) since $
S^{-i}AS=-iS^{-1}i\cdot AS=B$ $\
\Leftrightarrow\ $ $S^{-1}iAS=iB$.

(ii)$\Leftrightarrow$(iii) since $
i^{-i}iBi=Bi$.

(iii)$\Leftrightarrow$(iv) since $
i^{-i}iAi=Ai$.
\end{proof}

\begin{theorem}[cf. {\cite[Theorem 3]{hua}}]
\label{lsd} Each square quaternion
matrix $A$ is $\sigma$-consimilar
$(\sigma \in\{1,i\})$ to a complex
matrix that is a direct sum, uniquely
determined up to permutation of
summands, of Jordan blocks
\[
J_k(a+bi),\qquad a,b\in\mathbb R,\
\begin{cases}
b\ge 0&\text{if }\sigma =1,\\
a\ge 0&\text{if }\sigma =i.
\end{cases}
\]
\end{theorem}

\begin{proof}
Wiegmann \cite[Theorem 1]{wie} proved (see also Zhang's survey \cite[Theorem 6.4]{zha}) that each square
quaternion matrix $A$ is similar to a
direct sum, uniquely determined up to
permutation of summands, of Jordan
blocks $J_k(a+bi),\ a,b\in\mathbb R,\
b\ge 0$. Using Lemma \ref{lit}, we get
the desired Jordan canonical form of
quaternion matrices for
$i$-consimilarity.
\end{proof}

\section{Quaternion matrix equations
$\pmb{AX-\hat XB=C}$ and $\pmb{X-A\hat
XB=C}$}\label{jhf}

In this section, we consider the quaternion matrix
equations
\begin{equation*}\label{jry}
AX-\hat XB=C,\qquad X-A\hat
XB=C,
\end{equation*}
in which $A\in \HH^{m\times m}$, $B\in
\HH^{n\times n}$, $C\in \HH^{m\times
n}$ and $h\mapsto \hat h$ is an
arbitrary involutive automorphism of
$\HH$. These equations for the
automorphism \eqref{mjd} are studied in
\cite{jia,son1,son,song,yua} mainly by
means of the replacement of $p\times q$
quaternion matrices with the
corresponding $2p\times 2q$ complex
matrices or $4p\times 4q$ real
matrices. We  use the consimilarity
canonical form from Theorem \ref{lsd}
(in a similar way, Bevis, Hall, and
Hartwig \cite{bev} studied the complex
matrix equation $A\bar X - XB =C$ using
the consimilarity canonical form of
complex matrices).

By Lemma \ref{lih}, we can suppose that
$\hat h=h^{\sigma}$ for some $\sigma
\in\{1,i\}$ and all $h\in\HH$, and get
the equations
\begin{equation}\label{jjt}
AX-X^{\sigma}B=C,\qquad X-A
X^{\sigma}B=C\quad (\sigma
\in\{1,i\}).
\end{equation}
For all nonsingular $S\in \HH^{m\times
m}$ and $R\in \HH^{n\times n}$, these
equations are equivalent to
\[S^{-\sigma}ASS^{-1}XR
-S^{-\sigma}X^{\sigma}R^{\sigma}
R^{-\sigma}BR=S^{-\sigma}CR,\]
\[
S^{-\sigma}XR-S^{-\sigma}ASS^{-1}
X^{\sigma}R^{\sigma}R^{-\sigma}BR
=S^{-\sigma}CR.
\]
Taking $S$ and $R$ such that
$S^{-\sigma}AS$ and $R^{-\sigma}BR$ are
the complex canonical forms of $A$ and
$B$ determined by Theorem \ref{lsd}, we obtain
the matrix equations that are
considered in the following simple
theorem and its corollaries.

\begin{theorem}\label{kus}
Let the quaternion matrix equations
\eqref{jjt} be given by complex
matrices $A$ and $B$ and a quaternion
matrix $C=C_1+C_2j$ $(C_1,C_2\in\CC^{m\times
n})$. Then the sets of solutions of the equations \eqref{jjt} consist
of all matrices $X=X_1+X_2j$
in which $X_1$
and $X_2$ are complex matrices satisfying
\begin{equation}\label{grp}
AX_1-X_1B=C_1,\qquad AX_2-\sigma^2
X_2\bar B=C_2
\end{equation}
for the first equation in \eqref{jjt}
and, respectively,
\begin{equation}\label{fwe}
X_1-AX_1B=C_1,\qquad
X_2-\sigma^2AX_2\bar B=C_2
\end{equation}
for the second equation in \eqref{jjt}.
\end{theorem}

\begin{proof}
Write $X=X_1+X_2j$ and $C=C_1+C_2j$
$(X_1,X_2,C_1,C_2\in\CC^{m\times n})$. Then
$X^{i}=X_1-X_2j$,
$X^{\sigma}=X_1+\sigma^2X_2j$, and the
equations \eqref{jjt} take the form
\begin{equation}\label{llh}
\begin{matrix}
A(X_1+X_2j)-(X_1+\sigma^2X_2j)B=C_1+C_2j,\\
(X_1+X_2j)-A(X_1+\sigma^2X_2j)B=C_1+C_2j.
\end{matrix}
\end{equation}
Since $A$ and $B$ are complex matrices
and $jB=\bar Bj$, the quaternion matrix
equations \eqref{llh} are partitioned
into two pairs of complex matrix
equations \eqref{grp} and \eqref{fwe}.
\end{proof}

The theory of complex matrix equations
$AX-XB=C$ and $X-AXB=C$ was developed
in \cite[Chapter 18]{dym},
\cite[Chapter VIII]{gan}, \cite[Section
4.4]{j-h1}, \cite[Sections 12.3 and
12.5]{lan}, \cite{jia1,lan1,lan2}, and
in many other books and articles.

\begin{corollary}
\label{kus1} Let us apply Theorem
\ref{kus} to the
    quaternion matrix equation
$AX-X^{\sigma}B=C$ with $\sigma
\in\{1,i\}$. Let
\begin{equation}\label{lir}
A=J_{k_1}(\lambda _1)\oplus\dots\oplus
J_{k_p}(\lambda _p),\quad
B=J_{l_1}(\mu_1)\oplus\dots\oplus
J_{l_q}(\mu_q)\quad (\text{all }\lambda_i,\mu_j\in\CC)
\end{equation}
be two complex Jordan matrices $($for instance, canonical forms from Theorem
\ref{lsd} of quaternion matrices with respect to
$\sigma$-consimilarity$)$.
Write
\begin{equation}\label{lgt}
{\cal M}_{\sigma}:=\{\lambda
_1\dots,\lambda _p\}\cap\{\mu_1,\dots,\mu_q,
\sigma^2\bar\mu_1,\dots,\sigma^2\bar\mu_q\}.
\end{equation}

\begin{itemize}
\item[\rm(a)]
The following statements hold:
\begin{itemize}
\item[$\bullet$] If ${\cal
    M}_{\sigma}=\varnothing$, then
    $AX-X^{\sigma}B=C$ has a unique
    solution.

\item[$\bullet$] If ${\cal
    M}_{\sigma}\ne\varnothing$, then two
    cases may arise: either
    $AX-X^{\sigma}B=C$ has no
    solutions, or the set of its
    solutions is infinite and
    consists of all matrices
    $X_{\circ}+Y$ in which
    $X_{\circ}$ is a fixed
    particular solution of
    $AX-X^{\sigma}B=C$ and $Y$
    runs over all solutions of
    $AY-Y^{\sigma}B=0$.
\end{itemize}

\item[\rm(b)]
The set of solutions of $AY-Y^{\sigma}B=0$ is described as follows. Let us partition $Y$ into blocks in
    accordance with the partitions
    of $A$ and $B$:
\begin{equation}\label{mdu}
    Y=[Y_{\alpha \beta }]_{\alpha
    =1}^p{}_{\beta=1}^q,\qquad Y_{\alpha \beta }
    \text{ is $k_{\alpha }\times l_{\beta }.$}
\end{equation}
Then the set of solutions of
$AY-Y^{\sigma}B=0$ consists of all
quaternion matrices of the form $U+Vj$,
in which
\begin{itemize}
  \item[$\bullet$] $U=[U_{\alpha \beta
      }]_{\alpha
      =1}^p{}_{\beta=1}^q$ and
      $V=[V_{\alpha \beta
      }]_{\alpha
      =1}^p{}_{\beta=1}^q$ are
      complex matrices that are
      partitioned conformally to
      $Y$,
  \item[$\bullet$] $U_{\alpha \beta }=0$ if
      $\lambda _{\alpha}\ne \mu
      _{\beta }$,
  \item[$\bullet$] $V_{\alpha \beta }=0$ if
      $\lambda _{\alpha}\ne
      {\sigma}^2\bar\mu _{\beta }$,
 \item[$\bullet$] the other $U_{\alpha \beta
     }$ and $V_{\alpha \beta }$
     have the form
\[
  \begin{bmatrix}
   &a&b&\ddots&d \\
  &&a&\ddots&\ddots \\
&&&\ddots&b \\
  0&&&&a \\
  \end{bmatrix}\quad\text{or}
  \quad
\begin{bmatrix}
   a&b&\ddots&d \\
  &a&\ddots&\ddots \\
&&\ddots&b \\
  &&&a \\0
  \end{bmatrix}\quad (a,b,\dots,d\in\CC)
\]
if the number of rows is less than
or equal to the number of columns
or, respectively, the number of
rows is greater than the number of
columns $($we write the
off-diagonal units in all Jordan
blocks over the diagonal; see \eqref{kyr}$)$.
\end{itemize}
\end{itemize}
\end{corollary}

\begin{proof}
The statement (a) follows from
Theorem \ref{kus} since by
\cite[Chapter VIII, \S\,3]{gan} a
complex matrix equation $MX-XN=P$ with square matrices $M$ and $N$ has a
unique complex solution if and only if
$M$ and $N$ do not have common
eigenvalues; otherwise it is
contradictory or it has an infinite
number of solutions.

The statement (b) follows from Theorem
\ref{kus} and from the description in
\cite[Chapter VIII, \S\,1]{gan} of the
set of all complex matrices $Y$ such
that $MY=YN$ in which $M$ and $N$ are
complex Jordan matrices.
\end{proof}

\begin{example}
Let us consider the quaternion matrix equation
\begin{equation}\label{nfh}
\begin{bmatrix}
    0&0\\0&i
  \end{bmatrix}X-X^{\sigma}\begin{bmatrix}
    i&1\\0&i
  \end{bmatrix}=\begin{bmatrix}
    -k&j\\0&0
  \end{bmatrix}\qquad (\sigma
\in\{1,i\}).
\end{equation}
Its set \eqref{lgt} is $\{i\}$. Its particular solution is
\[
\begin{bmatrix}
    -j&0\\0&0
  \end{bmatrix} \text{ if }\sigma =1,\qquad
\begin{bmatrix}
    j&0\\0&j
  \end{bmatrix} \text{ if }\sigma =i.
\]
To solve the corresponding equation in which the right-hand part is zero:
\begin{equation}\label{dgt}
\begin{bmatrix}0&0\\0&i
 \end{bmatrix}Y-
  Y^{\sigma}\begin{bmatrix}
    i&1\\0&i
 \end{bmatrix}=
\begin{bmatrix}
    0&0\\0&0
 \end{bmatrix},
\end{equation}
we write $Y=U+Vj$ ($U$ and $V$ are complex matrices) and partition it as in \eqref{mdu}:
\[
Y=\begin{bmatrix}
    Y_{11}\\\hline Y_{21}
  \end{bmatrix}=
\begin{bmatrix}
    U_{11}\\\hline U_{21}
  \end{bmatrix}+
\begin{bmatrix}
    V_{11}\\\hline V_{21}
  \end{bmatrix}j.
\]
Replacing these blocks by the corresponding pairs of eigenvalues (in the notation of Corollary \ref{kus1}(b)), we get
\[
\begin{bmatrix}
    U_{11}\\\hline U_{21}
  \end{bmatrix}
\ \mapsto\
\begin{bmatrix}
    \lambda _1,\ \mu _1\\\hline
    \lambda _2,\ \mu _1
  \end{bmatrix}=
\begin{bmatrix}
    0,\ i\\\hline
    i,\ i
  \end{bmatrix},
\qquad
\begin{bmatrix}
    V_{11}\\\hline V_{21}
  \end{bmatrix}
\ \mapsto\
\begin{bmatrix}
    \lambda _1,\ \delta ^2\bar \mu _1\\\hline
    \lambda _2,\ \delta ^2\bar \mu _1
  \end{bmatrix}=
\begin{bmatrix}
    0,\ -\sigma ^2i\\\hline
    i,\ -\sigma ^2i
  \end{bmatrix}.
\]
By Corollary  \ref{kus1}, the set of solutions of \eqref{dgt} consists of all matrices
\[
\begin{bmatrix}
    0&0\\\hline 0&*
  \end{bmatrix}+
\begin{bmatrix}
    0&0\\\hline 0&0
  \end{bmatrix}j\
\text{ if }\sigma =1,\qquad\begin{bmatrix}
    0&0\\\hline 0&*
  \end{bmatrix}+
\begin{bmatrix}
    0&0\\\hline 0&*
  \end{bmatrix}j \ \text{ if }\sigma =i,
\]
in which the stars denote complex numbers,
and the set of solutions of \eqref{nfh} consists of all matrices
\[
X=\begin{cases}
\begin{bmatrix}
    -j&0\\0&0
  \end{bmatrix}+\begin{bmatrix}
    0&0\\0&c
  \end{bmatrix}\ (c\in\CC)& \text{if }\sigma =1,\\[15pt]
\begin{bmatrix}
    j&0\\0&j
  \end{bmatrix}+\begin{bmatrix}
    0&0\\0&h
  \end{bmatrix}\ (h\in\HH)& \text{if }\sigma =i.
\end{cases}
\]
\end{example}

\begin{corollary}
\label{kusu}
Let us apply Theorem
\ref{kus} to the
    quaternion matrix equation
$X-AX^{\sigma}B=C$  in which $\sigma
\in\{1,i\}$ and $A,B$ are the matrices \eqref{lir}.
Write
\begin{equation*}\label{krqs}
{\cal M}_{\sigma}:=\{\lambda^{-1}
_1\dots,\lambda _p^{-1}\}\cap\{\mu_1,\dots,
\mu_q,
{\sigma}^2\bar\mu_1,\dots,
{\sigma}^2\bar\mu_q\}
\end{equation*}
in which $0^{-1}:=\infty$.

The following statements hold:
\begin{itemize}
\item[$\bullet$] If ${\cal
    M}_{\sigma}=\varnothing$, then
    $X-AX^{\sigma}B=C$ has a unique
    solution.

\item[$\bullet$] If ${\cal
    M}_{\sigma}\ne\varnothing$, then two
    cases may arise: either
    $X-AX^{\sigma}B=C$ has no
    solutions, or the set of its
    solutions is infinite and
    consists of all matrices
    $X_{\circ}+Y$ in which
    $X_{\circ}$ is a fixed
    particular solution of
    $X-AX^{\sigma}B=C$ and $Y$
    runs over all solutions of
    $Y-AY^{\sigma}B=0$.
\end{itemize}
\end{corollary}

\begin{proof} These statements hold
since by \cite[Theorem 18.2]{dym} a
complex matrix equation $X-MXN=P$ with square $M$ and $N$ has a
unique complex solution if and only if
$\lambda \mu \ne 1$ for each eigenvalue
$\lambda $ of $M$ and for each
eigenvalue $\mu$ of $N$.
\end{proof}

\bigskip

\noindent{\bf Conclusion:} In all the papers that we know, the consimilarity of quaternion matrices is defined via the automorphism $h\mapsto -jhj$. We show that the consimilarity defined via $h\mapsto -ihi$ is more convenient for studying. All consimilarities of quaternion matrices defined via involutive automorphisms are reduced each other by reselecting of orthogonal
imaginary units $i,j,k$ in $\HH$.
\bigskip

\noindent{\bf Acknowledgment:}
The authors gratefully acknowledge the referees for their constructive advices.

\end{document}